\documentclass[12pt]{amsart}

\usepackage[usenames]{color}
\usepackage{amssymb}
\usepackage{amsmath}
\usepackage{amsthm}
\usepackage{amsfonts}
\usepackage{amscd}
\usepackage{graphicx}

\usepackage[colorlinks=true,
linkcolor=webgreen,
filecolor=webbrown,
citecolor=webgreen]{hyperref}

\makeatletter
\@namedef{subjclassname@2020}{\textup{2020} Mathematics Subject Classification}
\makeatother
\definecolor{webgreen}{rgb}{0,.5,0}
\definecolor{webbrown}{rgb}{.6,0,0}

\usepackage{color}
\usepackage{fullpage}
\usepackage{float}

\usepackage{graphics}
\usepackage{latexsym}
\usepackage{epsf}
\usepackage{breakurl}

\DeclareMathOperator{\per}{per}

\begin{document}

\theoremstyle{plain}
\newtheorem{theorem}{Theorem}
\newtheorem{corollary}[theorem]{Corollary}
\newtheorem{lemma}[theorem]{Lemma}
\newtheorem{proposition}[theorem]{Proposition}

\theoremstyle{definition}
\newtheorem{definition}[theorem]{Definition}
\newtheorem{example}[theorem]{Example}
\newtheorem{conjecture}[theorem]{Conjecture}

\theoremstyle{remark}
\newtheorem{remark}[theorem]{Remark}

\title{Consecutive Power Occurrences in Sturmian Words}

\author{Jason Bell}
\address{Jason Bell\\
University of Waterloo\\
Department of Pure Mathematics\\
200 University Avenue West\\
Waterloo, Ontario \  N2L 3G1\\
Canada}
\email{jpbell@uwaterloo.ca}

\author{Chris Schulz}
\address{University of Waterloo\\
Department of Pure Mathematics\\
200 University Avenue West\\
Waterloo, Ontario \  N2L 3G1\\
Canada}
\email{chris.schulz@uwaterloo.ca}

\author{Jeffrey Shallit}
\address{University of Waterloo\\
School of Computer Science\\
200 University Avenue West\\
Waterloo, Ontario \  N2L 3G1\\
Canada}
\email{shallit@uwaterloo.ca}

\thanks{The authors were supported by Discovery Grants from the National Sciences and Engineering Research Council of Canada.}


\title{Consecutive power occurrences in Sturmian words}

\maketitle

\begin{abstract}
We show that every Sturmian word has the property that the distance between consecutive ending positions of cubes occurring in the word is always bounded by $10$ and this bound is optimal, extending a result of Rampersad, who proved that the bound $9$ holds for the Fibonacci word.  We then give a general result showing that for every $e \in [1,(5+\sqrt{5})/2)$ there is a natural number $N$, depending only on $e$, such that every Sturmian word has the property that the distance between consecutive ending positions of $e$-powers occurring in the word is uniformly bounded by $N$. \end{abstract}

\section{Introduction}

In this note we are concerned with the occurrences of powers in words.   
We say that a length-$n$ word $w = w[0..n-1]$ has \emph{period} $p$ if
$w[i]=w[i+p]$ for $i\in \{0,\ldots ,n-p-1\}$.  The smallest such
period is called {\it the\/} period and is written
$\per(w)$.   The exponent of a length-$n$ word $w$ is
defined to be $\exp(w) = n/\per(w)$. 
For example, $\exp({\tt entente}) = 7/3$.  For a real number $e \geq 1$, we say a word $w$
is an $e$-\emph{power} if $\lceil e \per(w) \rceil$ is equal to $|w|$, the length of $w$.  In particular, a finite word being an $e$-power and having exponent $e$ are not in general the same: the exponent is necessarily rational, while there exist $e$-powers for every real number $e\ge 1$.
A $2$-power is called a {\it square}, like the English word {\tt murmur}.  A $3$-power is called a {\it cube}, like the English word
{\tt shshsh}.

Let $\alpha = (1+\sqrt{5})/2$ be the golden ratio and let
$\bf x$ be any infinite word.
In 1997 the third author of this paper conjectured,
and Mignosi, Restivo, and Salemi
later proved \cite{Mignosi&Restivo&Salemi:1998}, that
every sufficiently long prefix of $\bf x$
has a suffix of exponent $\geq \alpha^2$ 
if and only if $\bf x$ is ultimately
periodic.  This is an example of the ``local periodicity implies
global periodicity'' phenomenon.

Furthermore, the constant $\alpha^2$ is best possible.  
Let ${\bf f} = {\bf f}[0] {\bf f}[1] {\bf f}[2] \cdots = 01001010\cdots$ denote the infinite Fibonacci word \cite{Berstel:1986b}, the fixed
point of the morphism sending $0 \rightarrow 01$ and $1 \rightarrow 0$.
Mignosi, Restivo, and Salemi  \cite{Mignosi&Restivo&Salemi:1998}
also showed that for all $\epsilon > 0$ every sufficiently long
prefix of $\bf f$ has a suffix that is an $\alpha^2 - \epsilon$ power.  For example, there is a square suffix of every prefix of length $n \geq 6$ in {\bf f}, and there is 
a ${5\over2}$-power suffix of every prefix of length $n \geq 220$ in $\bf f$.
An explicit version of this result was recently proved
by the third author \cite{Shallit:2023b}.

More generally, given an infinite word $\bf x$ and
an exponent $e$, one can consider
the list of all positive integers $p_1 < p_2 < \cdots$ marking
the position where some $e$-power in
$\bf x$ ends.  For example, for the Fibonacci word we have
\begin{itemize}
\item $(010)^3$ ends at position $13$
\item $(01001)^3$ ends at position $22$
\item $(10010)^3$ ends at position $23$
\item $(010)^3$ ends at position $26$
\item $(010)^3$ ends at position $34$
\end{itemize}
and so forth.

Recently Rampersad \cite{Rampersad:2023} obtained an explicit description
of the positions where cubes end in the Fibonacci word $\bf f$.

In particular, he observed that if $p_1 < p_2$ are two consecutive
positions where cubes end in $\bf f$, then $p_2 - p_1 \leq 9$, and 
more precisely $p_2 - p_1 \in \{ 1,2,3,4,8,9\}$.  
(Note that he phrased his discussion in terms of the lengths of {\it runs\/} of consecutive positions where there are {\it no cubes ending}, so the numbers
he presented differ by $1$ from ours.)

Since the Fibonacci word is the simplest of a much larger class of binary infinite words
called Sturmian words, this naturally raises the question of whether
a similar result holds for this much larger class.
Sturmian words have several equivalent descriptions, as follows:
\begin{itemize}
\item[(a)]  Infinite words of the form
$(\lfloor (n+1) \gamma + \beta \rfloor - \lfloor n \gamma + \beta \rfloor)_{n \geq 1}$
for real $0<\gamma < 1$, $0 \leq \beta < 1$, where $\gamma$ is irrational.   Here $\gamma$ is called the
{\it slope\/} of the word and $\beta$ is called the \emph{intercept}.

\item[(b)] Infinite words having exactly $n+1$ distinct factors
of length $n$ for all $n \geq 0$.  Here by a {\it factor\/} we mean a consecutive block lying inside the word.

\item[(c)] Infinite aperiodic binary balanced words, that is, words such that
for all factors $x,y$ of the same length, and all $a \in \{0,1\}$,
the inequality $\left| |x|_a - |y|_a \right| \leq 1$ holds, where $|z|_a$ denotes the number of occurrences of the letter $a$ in $z$.
\end{itemize}
For more information about this class of words, see, for
example, \cite[Chap.~2]{Lothaire:2002}.

In this note we generalize Rampersad's result to all Sturmian words.

\begin{theorem}
Let $\bf x$ be a Sturmian word.  Then the maximum gap between positions where cubes end in $\bf x$ is at most $10$, and this bound is optimal. 
\label{thm1}
\end{theorem}
In fact, our proof shows that the gap between consecutive occurrences of cubes having period at most $5$ is at most $10$ for every Sturmian word; furthermore, the proof shows that the optimal bound $10$ is achieved by the Sturmian
characteristic word with slope $\sqrt{2}-1$ and intercept $0$.

It is natural to ask whether a similar result to Theorem \ref{thm1} holds for powers other than cubes. It is known that the Fibonacci word does not have $\beta$-powers for $\beta=(5+\sqrt{5})/2$, and so it is necessary to restrict to exponents that are less than $\beta$ when considering the gap question over all Sturmian words. Once this condition is imposed, however, we are able to again prove a general result about the existence of uniform bounds on gaps of $e$-powers in Sturmian words for every $e<\beta$.

\begin{theorem}
Let $\bf x$ be a Sturmian word and let $e\in [1,(5+\sqrt{5})/2)$. Then there are natural numbers $N$ and $k$, depending only on $e$, such that the maximum gap between the ending positions of $e$-powers of
period at most $k$ in $\bf x$ is at most $N$.
\label{thm:genexp}
\end{theorem}

\section{Proof of Theorem~\ref{thm1}}

We begin with a lemma.
As it turns out, it suffices to consider words of bounded period.

\begin{lemma}
Every balanced binary word of length $17$ contains a cube of period
at most $5$ in it, and the bound $17$ is best possible.\label{lem:1}
\end{lemma}

\begin{proof}
We enumerate all $594$ binary balanced words of length $17$ and check.
For length $16$, the word $0010100101001001$ is balanced but contains no cube.
\end{proof}

\begin{proof}[Proof of Theorem 1]
We claim that every balanced binary word of length $32$ 
 contains at least two consecutive cube occurrences, each of period at most $5$.    From Lemma \ref{lem:1}, we know the prefix of length 17 contains a cube of period 5, and so does the
suffix of length 17.   Since in a word of length $32$ these two cubes can overlap in at most two symbols, they must be distinct cube occurrences.  Thus the maximum possible gap must appear in some balanced word of length $32$.   It then suffices to examine all balanced words of this length (there are 3650) and compute all gaps
between consecutive ending positions of cubes of period
$5$ in all of them.  The longest is $10$.

To prove that $10$ is optimal, we use the theorem-prover {\tt Walnut}; see \cite{Mousavi:2016,Shallit:2023} for more details about it. We now consider the Sturmian
characteristic word with slope $\sqrt{2}-1$ and intercept $0$. This word has an associated Ostrowski numeration system based on the Pell numbers, $1,2,5,12,29,70,169,\ldots$; this system is built-in to {\tt Walnut} and is invoked by saying {\tt msd\_pell} (see
\cite{Baranwal:2020,
Baranwal&Schaeffer&Shallit:2021,
Baranwal&Shallit:2019}). We now use the following {\tt Walnut} code:

\begin{verbatim}
reg odd0 msd_pell "(0|1|2)*(1|2)0(00)*":
def root2 "?msd_pell (n>=0) & $odd0(n+1)":
combine RT2 root2:
# make a DFAO RT2 for Sturmian sequence w/slope sqrt(2)-1 

def end_in_cube "?msd_pell Ei,j (j>=1) & (i+3*j=n) & 
   At (t<2*j) => RT2[i+t]=RT2[i+j+t]":
# is there a cube ending at position n

def two_consec "?msd_pell (i<j) & $end_in_cube(i) & 
   $end_in_cube(j) & At (i<t & t<j) => ~$end_in_cube(t)":
# are i and j two consecutive cube-ending positions?

def two_consec_dist "?msd_pell (n>0) & Ei,j i+n=j & $two_consec(i,j)":
# it accepts only 1, 110 and 200, i.e., 1, 7 and 10 in Pell
\end{verbatim}
The code first creates a deterministic finite-state automaton with output, named {\tt RT2}, for the Sturmian sequence $\bf x$ with slope $\sqrt{2}-1$ and intercept $0$. It next determines the set of natural numbers $n$ for which there is a cube occurring in $\bf x$ that ends at position $n$. Finally, the code finds the set of gaps between consecutive ending positions of cubes and determines that the only gaps that occur are the natural numbers whose expansions in the Pell numeration system are $1, 110,$ and $200$. These are the numbers $1, 7, 10$ and so the bound of $10$ is best possible.
\end{proof}
\begin{remark} We note that the proof in fact shows the slightly stronger claim that the gap between consecutive occurrences of cubes of period at most $5$ is at most $10$ for every Sturmian word and that the bound $10$ is achieved by the Sturmian
characteristic word with slope $\sqrt{2}-1$ and intercept $0$.
\end{remark}
\section{General exponents}

We now give the proof of Theorem \ref{thm:genexp}. We recall that a {\it partial function\/} $f$ from a set $X$ to a set $Y$ is simply a map defined on a non-empty subset $U$ of $X$ that maps $U$ into the set $Y$.  In the case that the map is defined on all of $X$, the map is called a {\it total function}.  We write $f:X\rightharpoonup Y$ for a partial function from $X$ to $Y$.

Given two partial functions $f,g: \mathbb{N}\rightharpoonup \mathbb{R}_{\ge 0}$, we declare that $f\preceq g$ if $f(i)\le g(i)$ for all $i$ in the domain of both $f$ and $g$.  We note that $\preceq$ is not a partial order on partial functions, but it is, however, a partial order on total functions.  We recall that for a right-infinite word $\bf x$ over a finite alphabet, we have a subword complexity function $p_{{\bf x}}$ whose value at $n$ is the number of distinct length-$n$ factors of ${\bf x}$.  If $w$ is a finite word, we have a partial function $p_w:\mathbb{N}\rightharpoonup \mathbb{N}$ that is defined on the nonnegative integers that are less than or equal to the length of $w$ and whose value at $i$ is the number of factors of $w$ of length $i$ for all $i$ less than or equal to the length of $w$.

\begin{lemma}
Let $\Sigma$ be a finite alphabet, let $h:\mathbb{N}\to \mathbb{N}$ be a total function, and let $S$ be a non-empty set of words over $\Sigma$.  Suppose that every uniformly recurrent word ${\bf x}$ with $p_{\bf x}\preceq h$ has the property that ${\bf x}$ contains some word from $S$ as a factor.
Then there exists a finite subset $S_0$ of $S$ and a natural number $N$,
depending only on $h$ and $S_0$,
such that, whenever ${\bf x}$ is a right-infinite word with $p_{\bf x}\preceq h$, every factor of ${\bf x}$ of length $N$ contains a factor from $S_0$.
\label{lem:h}
\end{lemma}
\begin{proof}
Let $U_n$ denote the set of words $u$ over $\Sigma$ of length $n$ such that $p_u\preceq h$ and such that $u$ does not have an element of $S$ as a factor.  Then $U_n$ is closed under the process of taking factors for all $n\ge 1$.  If there is some $N\ge 1$ such that $U_N$ is empty, then letting $S_0$ be the set of words in $S$ of length at most $N$, we see that
whenever ${\bf x}$ is a right-infinite word with $p_{\bf x}\preceq h$, every factor of ${\bf x}$ of length $N$ contains a factor from $S_0$.  

Thus we may assume that each $U_n$ is non-empty.  
By K{\"o}nig's infinity lemma
(see \cite{Konig:1927} or \cite[\S 2.3.4.3]{Knuth:1997}), there is a right-infinite word ${\bf x}$ with the property that all of its finite factors are in the union of the $U_i$.  In particular, $p_{\bf x}\preceq h$ and ${\bf x}$ has no words from $S$ as a factor.  Then by Furstenberg's theorem \cite{Furstenberg:1981}, there is a uniformly recurrent right-infinite word ${\bf y}$ whose finite factors are all factors of ${\bf x}$, so $p_{\bf y}\preceq p_{\bf x}\preceq h$ and ${\bf y}$ has no occurrences of factors from $S$, a contradiction.
\end{proof}

\begin{proof}[Proof of Theorem \ref{thm:genexp}]
Let $h(n)=n+1$ for $n\ge 0$ and let $P_{e}$ denote the set of words $w$ over $\Sigma$ that are $e$-powers. By work of Damanik and Lenz
\cite{Damanik&Lenz:2002}, every Sturmian word
has an element from $P_{e}$ as a factor.
Since every uniformly recurrent word whose subword complexity function is $\preceq h$ is either Sturmian or periodic and since periodic words always contain elements from $P_{e}$ as factors, Lemma \ref{lem:h} gives that there is a finite subset $Q_{e}$ of $P_{e}$ and some number $n=n(e)$ such that every length-$n$ factor of a Sturmian word $\bf x$ has an occurrence of a word from $Q_{e}$.  In particular, if we take $k$ to be the supremum of periods in $Q_{e}$ and $N=2n$, we see that gaps between consecutive $e$-powers of period at most $k$ are bounded by $N$.  
\end{proof}

\section{Other exponents}
When one considers Theorem \ref{thm:genexp}, a natural question that arises is whether one can bound the associated quantities $k$ and $N$ explicitly in terms of the exponent $e\in [1,(5+\sqrt{5})/2)$. This appears to be somewhat subtle in general, but using the same computational approach as we employed in establishing Theorem \ref{thm1}, we can obtain results for several other exponents.

For each exponent $e$ we report
the smallest $n$ such that such that every balanced word of length
$n$ contains a factor that is an $e$-power, the largest period $p$ of an $e$-power in those words, the largest gap $g$ between ending positions of $e$-powers, and a quadratic irrational $\gamma$ realizing this
gap $g$.
\begin{table}[H]
    \centering
    \begin{tabular}{c|c|c|c|c}
        $e$ & $n$ & $p$ & $g$ & $\gamma$  \\
        \hline
        $5/2$ & $9$ & $3$ & $6$ & $(5+\sqrt{5})/10$ \\
        $8/3$ & $15$ & $5$ & $9$ & $\sqrt{2}-1$ \\
        $3$ & $17$ & $5$ & $10$ & $\sqrt{2}-1$ \\
        $16/5$ & $30$ & $8$ &$17$ & $(25-\sqrt{5})/62$ \\
        $23/7$ & $50$ & $13$ & $27$ & $(59+\sqrt{5})/158$\\
        $10/3$ & $69$ & $18$ & $37$ & $(217-\sqrt{5})/298$\\
    \end{tabular}
    \caption{Data for some exponents}
    \label{tab1}
\end{table}
It would be interesting to better understand the relationship between the largest gap $g$ and the exponent $e$ as $e$ approaches $(5+\sqrt{5})/2$.

\end{document}